\documentclass{article}%
\usepackage{amsmath}
\usepackage{amsfonts}
\usepackage{amssymb}
\usepackage{graphicx}%
\providecommand{\U}[1]{\protect\rule{.1in}{.1in}}
\newtheorem{theorem}{Theorem}

\newtheorem{lemma}[theorem]{Lemma}

\newtheorem{remark}[theorem]{Remark}

\newenvironment{proof}[1][Proof]{\noindent\textbf{#1.} }{\ \rule{0.5em}{0.5em}}
\begin{document}

\title{Some noteworthy alternating trilinear forms}
\author{Jan Draisma\thanks{Jan Draisma is supported by a Vidi grant from the
Netherlands Organisation for Scientific Research (NWO)} ~and Ron Shaw}
\date{}
\maketitle

\begin{abstract}
Given an alternating trilinear form $T\in\operatorname{Alt}(\times^{3}V_{n})$
on $V_{n}=V(n,\mathbb{F})$ let $\mathcal{L}_{T}$ denote the set of
$T$-singular lines in $\operatorname*{PG}(n-1)=\mathbb{P}V_{n},$ consisting
that is of those lines $\langle a,b\rangle$ of $\operatorname*{PG}(n-1)$ such
that $T(a,b,x)=0$ for all $x\in V_{n}.$ Amongst the immense profusion of
different kinds of $T$ we single out a few which we deem noteworthy by virtue
of the special nature of their set $\mathcal{L}_{T}.$

\end{abstract}



\section{Introduction \label{Sec Intro}}

We will deal with a finite-dimensional vector space $V_{n}=V(n,\mathbb{F})$
and the associated projective space $\operatorname*{PG}(n-1,\mathbb{F}%
)=\mathbb{P}V_{n}.$ The $\tbinom{n}{2}$-dimensional space $\operatorname{Alt}%
(\times^{2}V_{n})$ consisting of alternating bilinear forms on $V_{n}$ is of
course very well understood. If $n=2m,$ or if $n=2m+1,$ then the nonzero
elements $B\in\operatorname{Alt}(\times^{2}V_{n})$ fall into $m$
$\operatorname*{GL}(n,\mathbb{F})$-orbits $\{\Omega_{k}\}_{k=1,2,...m},$ where
$\Omega_{k}$ consists of those $B$ which have rank $2k.$ For a given
$B\in\operatorname{Alt}(\times^{2}V_{n})$ a point $\langle a\rangle
\in\mathbb{P}V_{n}$ is said to be ($B$-)\emph{singular} whenever $B(a,x)=0$
holds for all $x\in V_{n}.$ Consequently if $n$ is odd then $B$-singular
points exist for any $B,$ while if $n=2m$ is even then only when $B$ is on the
orbit $\Omega_{m}$ do $B$-singular points not exist.

In the present paper we consider instead the $\tbinom{n}{3}$-dimensional space
$\operatorname{Alt}(\times^{3}V_{n})$ consisting of alternating trilinear
forms on $V_{n}$. In contrast with $\operatorname{Alt}(\times^{2}V_{n})$ the
mathematics of the space $\operatorname{Alt}(\times^{3}V_{n})$ is much more
complicated (and interesting!). In particular the orbit structure of
$\operatorname{Alt}(\times^{3}V_{n})$ is only known in certain low-dimensional
cases. Alternating trilinear forms have been classified in dimension $n\leq7$
over an arbitrary field, see \cite{Cohen88,Schouten31}, and also in dimension
$8$ over $\mathbb{C}$ and $\mathbb{R},$ see \cite{Djokovic83,Gurevich35}. Over
$\mathbb{C}$ there are 23 orbits in dimension $n=8,$ but in dimension $n=9$
the number of orbits is known to be infinite.

Over a finite field $\operatorname*{GF}(q)$ there are of course, in any finite
dimension $n,$ \textquotedblleft only\textquotedblright\ a finite number of
$\operatorname*{GL}(n,q)$-orbits. But in fact the number of orbits increases
extremely rapidly with increasing $n.$ To demonstrate this it will suffice to
use a crude upper bound on the order of the group $\operatorname*{GL}(n,q),$
namely $|\operatorname*{GL}(n,q)|\ll q^{n^{2}},$which holds on account of the
inclusion $\operatorname*{GL}(V_{n})\subset L(V_{n},V_{n}).$ Since $\wedge
^{3}V_{n,q}$ is of size $q^{n(n-1)(n-2)/6}$ it follows that $|\wedge
^{3}V_{n,q}|/|\operatorname*{GL}(n,q)|\gg q^{n(n^{2}-9n+2)/6}.$ In particular
for $n=10$ we have $|\wedge^{3}V_{10}|/|\operatorname*{GL}(10,q)|\gg q^{20},$
and so even on the ridiculous assumption that the stabilizer group of every
$T\in\wedge^{3}V_{10}$ is the whole of $\operatorname*{GL}(10,q)$ the number
of $\operatorname*{GL}(10,q)$-orbits would be substantially more than $q^{20}%
$. And for $n=20$ the number of $\operatorname*{GL}(20,q)$-orbits in
$\wedge^{3}V_{20}$ is much more than $q^{740}.$

The violence of the combinatorial explosion which takes place for $n>8$ is
really quite startling! This occurs even over the smallest fields. For on
setting $N(n,q)=|\wedge^{3}V_{n,q}|/|\operatorname*{GL}(n,q)|$ we find for
$q=2$ the following approximate values
\[%
\begin{tabular}
[c]{|clllllll|}\hline
$n=$ & 5 & 6 & 7 & 8 & 9 & 10 & 11\\
$N(n,2)$ & $0.00010$ & $0.000053$ & $0.00021$ & $0.0135$ & $27.6$ &
$3.6\times10^{6}$ & $6.1\times10^{13}$\\\hline
\end{tabular}
\ .
\]

Faced with this great multitude of orbits for alternating trilinear forms one
naturally hopes that there are a few orbits which are singled out by having
some special property and which thus deserve further attention. In the case of
alternating bilinear forms the outstanding $\operatorname*{GL}(n,\mathbb{F}%
)$-orbit occurs of course in even dimension $n=2m$ and consists of those
$B\in\operatorname{Alt}(\times^{2}V_{n})$ which have no singular points, the
stabilizer groups being $\cong\operatorname{Sp}(2m,\mathbb{F}).$ Now in the
case of $T\in\operatorname{Alt}(\times^{3}V_{n})$ one may define a point
$\langle a\rangle\in\mathbb{P}V_{n}$ to be $T$\emph{-singular} whenever
$T(a,x,y)=0$ holds for all $x,y\in V_{n}.$ Also one may define a subspace
$\operatorname*{rad}T$ of $V_{n}$ by%
\begin{equation}
\operatorname*{rad}T=(a\in V_{n}:T(a,x,y)=0\quad\text{for all }x,y\in V_{n}\}
\label{radT}%
\end{equation}
and call $T$ \emph{non-degenerate} whenever $\operatorname*{rad}T=\{0\}.$ But,
just as in the bilinear case, there is not much interest in degenerate $T,$
since one naturally switches one's attention to the non-degenerate trilinear
form induced in the lower-dimensional quotient space $V_{n}%
/\operatorname*{rad}T.$ However of crucial importance in the case of
$T\in\operatorname{Alt}(\times^{3}V_{n})$ are those projective lines $\langle
a,b\rangle$ in $\operatorname*{PG}(n-1)=\mathbb{P}V_{n}$ which are
$T$\emph{-singular}, satisfying that is
\begin{equation}
T(a,b,x)=0\quad\text{for all }x\in V_{n}. \label{T-singular line}%
\end{equation}
For a given $T\in\operatorname{Alt}(\times^{3}V_{n})$ we will denote by
$\mathcal{L}_{T}$ the set consisting of all the $T$-singular lines in
$\operatorname*{PG}(n-1,\mathbb{F})=\mathbb{P}V_{n}$.

\begin{remark}
The space $\operatorname{Alt}(\times^{3}V_{n})$ of alternating 3-forms is
naturally isomorphic to the space $\wedge^{3}V_{n}^{\ast}$ of dual trivectors,
and sometimes statements concerning an element $T\in\operatorname{Alt}%
(\times^{3}V_{n})$ will be phrased in terms of its isomorphic image
$t\in\wedge^{3}V_{n}^{\ast}.$ If $\{f_{i}\}_{1\leq i\leq n}$ is the basis for
$V_{n}^{\ast}$ dual to the basis $\{e_{i}\}_{1\leq i\leq n}$ for $V_{n}$ then,
on writing $f_{ijk}:=f_{i}\wedge f_{j}\wedge f_{k},$ we have%
\begin{equation}
t=%
{\textstyle\sum\nolimits_{1\leq i<j<k\leq n}}
c_{ijk}f_{ijk},\quad\text{where }c_{ijk}:=T(e_{i},e_{j},e_{k}).
\label{Iso t with T}%
\end{equation}
Equivalently expressed, each $t\in\wedge^{3}V_{n}^{\ast}$ gives rise to an
element $T\in$ $\operatorname{Alt}(\times^{3}V_{n})$ by way of%
\begin{equation}
T(x,y,z)=<t|x\wedge y\wedge z>, \label{Iso T with t}%
\end{equation}
where $<\cdot|\cdot>$ is the standard determinantal pairing of $\wedge
^{3}V_{n}^{\ast}$ with $\wedge^{3}V_{n}$ given by $<f_{1}\wedge f_{2}\wedge
f_{3}|v_{1}\wedge v_{2}\wedge v_{3}>=\det[f_{i}(v_{j})].$
\end{remark}

\section{Alternating 3-forms having no singular lines?
\label{Sec no singular lines}}

A question immediately arises: \emph{at least for some }$(n,\mathbb{F}%
),$\emph{ does there exist }$T\in\operatorname{Alt}(\times^{3}V_{n})$\emph{
such that }$\mathcal{L}_{T}$\emph{ is empty? }

Well, for any field $\mathbb{F},$ certainly not in even dimension $n=2m.$ For
given $T\in\operatorname{Alt}(\times^{3}V_{2m})$ choose any direct sum
decomposition $V_{2m}=\prec a\succ\oplus V_{2m-1}$ and consider the element
$B_{a}\in\operatorname{Alt}(\times^{2}V_{2m-1})$ defined by $B_{a}%
(x,y)=T(a,x,y).$ Since the dimension of $V_{2m-1}$ is odd there exists at
least one $B_{a}$-singular point $\langle b\rangle,$ whence $\langle
a,b\rangle$ is a $T$-singular line. It follows that \emph{through each point
}$\langle a\rangle\in\mathbb{P}V_{2m}$ \emph{there passes at least one }%
$T$\emph{-singular line. }

Concerning odd dimension $n,$ certainly a non-zero $T\in\operatorname{Alt}%
(\times^{3}V_{n})$ has no $T$-singular lines in the special case $n=3.$ But
\emph{for} $n>3$ \emph{if }$\mathbb{F}$\emph{ is quasi-algebraically closed
then it is known that }$T$\emph{-singular lines always exist: }see\emph{
\cite[Theorem 1.1]{DraismaShaw}. }

An important field not covered in this last statement is the real field
$\mathbb{R}.$ And in the case of a real $7$-dimensional space we have an
affirmative answer to our query: \emph{if }$V_{7}=V(7,\mathbb{R})$ \emph{there
exists }$T\in\operatorname{Alt}(\times^{3}V_{7})$ such that $\mathcal{L}_{T}%
$\emph{ is empty. }To see this, recall that in a real 7-dimensional Euclidean
space $V_{7}$ there exist, see \cite{BrownGray67}, bilinear vector cross
products $a\times b$ which satisfy \emph{(i)} $a\times b.a=0=a\times b.b$ and
\emph{(ii)} $a\times b.a\times b=(a.a)(b.b)-(a.b)^{2}.$ Then upon defining
$T(a,b,c)=a\times b,c$ it follows that $T\in\operatorname{Alt}(\times^{3}%
V_{7});$ moreover, since, by \emph{(ii), }$a\times b\neq0$ for linearly
independent $a,b,$ we see that $\mathcal{L}_{T}$ is empty. One such $T$ has
for its isomorphic image the dual trivector $t$ given by
\begin{equation}
t=f_{124}+f_{235}+f_{346}+f_{457}+f_{561}+f_{672}+f_{713}, \label{t_Fano}%
\end{equation}
where the presence of $+f_{ijk}$ in this expression for $t$ goes along with
the cross product relation $e_{i}\times e_{j}=+e_{k}.$

As is well known, the existence in real seven dimensions of these vector cross
products is related to the exceptional existence of the real division algebra
of the (non-split) octonions. The next theorem shows that the real dimension
$n=7$ is also exceptional since $T$-singular lines always exist for
$T\in\operatorname{Alt}(\times^{3}V_{n})$ in any other real dimension $n>3.$

\begin{theorem}
\label{Thm Except n = 3, 7}Except when $n\in\{3,7\}$ every alternating
trilinear form on a real vector space $V_{n},n>2,$ possesses singular lines.
\end{theorem}

\begin{proof}
Suppose that $T\in\operatorname{Alt}(\times^{3}V_{n})$ is such that
$\mathcal{L}_{T}$ is empty. Set $V_{n+1}=\mathbb{R}\oplus V_{n},$ and equip
$V_{n}$ with $\operatorname{O}(n)$-geometry by making a(ny) choice $x.y$ of a
positive definite scalar product on $V_{n}$. Extend this to a positive
definite scalar product $a.b$ on $V_{n+1}$ by defining
\begin{equation}
(\alpha,x).(\beta,y)=\alpha\beta+x.y~,\quad\alpha,\beta\in\mathbb{R},~x,y\in
V_{n}. \label{a.b}%
\end{equation}
Make $V_{n}$ into a real algebra by defining the algebra product of $x$ and
$y$ to be that element $x\times y\in V_{n}$ such that%
\begin{equation}
x\times y.z=T(x,y,z),\qquad\text{for all }z\in V_{n}. \label{x cross y}%
\end{equation}
Since $T$ is alternating it follows that
\begin{equation}
x\times y=-y\times x\in\prec x,y\succ^{\perp}. \label{property 1}%
\end{equation}
Further, since we are supposing that there are no $T$-singular lines, if $x$
and $y$ are linearly independent then%
\begin{equation}
x\times y\text{ is a \emph{nonzero }element of }V_{n}\text{ which is
perpendicular to the plane}\prec x,y\succ. \label{property 2}%
\end{equation}
We now make $V_{n+1}$ into a real algebra $\mathcal{A}$ by laying down that
$1\in\mathbb{R}$ is an identity element and that the $\mathcal{A}$-product of
$x,y\in V_{n}$ is%
\begin{equation}
xy=-x.y+x\times y,\qquad x,y\in V_{n}. \label{xy}%
\end{equation}
It follows from (\ref{property 1})-(\ref{xy}) that the algebra $\mathcal{A}$
has no zero divisors. Consequently, see \cite[Section II.2]{Schafer},
$\mathcal{A}$ is a division algebra over $\mathbb{R}.$ The theorem now follows
since real division algebras exist only in dimensions $1,2,4,8$: see
\cite{BottMilnor}, \cite{Kervaire}.
\end{proof}

\begin{remark}
\label{Rmk Another proof of no singular lines}In \cite[Theorem 3.2]%
{DraismaShaw} it was proved that, over any field $\mathbb{F},$ the union of
all lines $L\in\mathcal{L}_{T},~T\in\operatorname{Alt}(\times^{3}V_{2k+1}),$
is either the whole of $V_{2k+1}$ or is a hypersurface in $\operatorname*{PG}%
(2k)$ with equation $f_{T}(x)=0,~x\in V_{2k+1},$ where $f_{T}$ is a certain
homogeneous polynomial of degree $k-1$. In the case of a space $V_{7}$ the
polynomial $f_{T}$ has degree 2, and if $T$ has trivector $t$ as in
(\ref{t_Fano}) then (up to an overall sign) one finds that
\begin{equation}
f_{T}(x)=\sum\nolimits_{i=1}^{7}(x_{i})^{2}. \label{F_T = squares}%
\end{equation}
In the case $\mathbb{F}=\mathbb{R}$ it follows that $f_{T}(x)\neq0$ for all
$x\neq0.$ Thus we have obtained another proof that $\mathcal{L}_{T}$ is empty
if $t$ is as in (\ref{t_Fano}) and $V_{7}=V(7,\mathbb{R}).$
\end{remark}

\begin{remark}
\label{Rmk t_Fano for GF(q)}It is of some interest to consider $t$ as in
(\ref{t_Fano}) in the cases when $V_{7}=V(7,q).$

\emph{(i)} First suppose that $q=2^{h}$ is even. In which case $f_{T}%
(x)=(\sum\nolimits_{i=1}^{7}x_{i})^{2},$ whence the union of all the
$T$-singular lines is the hyperplane with equation $\sum\nolimits_{i=1}%
^{7}x_{i}=0.$

\emph{(ii)} If $q$ is odd then the union of all the $T$-singular lines is the
parabolic quadric $\mathcal{P7}_{6}$ in $\operatorname*{PG}(6,2)$ having
equation $\sum\nolimits_{i=1}^{7}(x_{i})^{2}=0.$

That the cases of even $q$ and odd $q$ are quite different is highlighted by
the fact that the alternating trilinear form given by $t=f_{124}%
+f_{235}+f_{346}+f_{457}+f_{561}+f_{672}+f_{713}$ can be seen to belong to the
same $\operatorname*{GL}(7,q)$-orbit as $f_{6}$ in \cite[Table 1]{Cohen88} if
$q$ is even, but to belong to the same $\operatorname*{GL}(7,q)$-orbit as
$f_{9}$ in \cite[Table 1]{Cohen88} if $q$ is odd. In particular, upon using
Mathematica to compute the quadratic form $f_{T}$ for all orbit
representatives in \cite[Table 1]{Cohen88}, we found that only $f_{9}$ gives
rise to a nonsingular quadratic form.
\end{remark}

\section{Alternating 3-forms yielding spreads in $\mathbb{P}V_{2m}?$
\label{Sec L_T a spread}}

\vspace{-0.2in}

As noted in the preceding section, \emph{if }$T$ \emph{is any alternating
3-form on an even-dimensional space }$V_{2m},$ \emph{over any field}
$\mathbb{F},$\emph{ then through each point }$\langle a\rangle\in
\mathbb{P}V_{2m}$ \emph{there passes at least one }$T$\emph{-singular line.}
In our attempt to find alternating 3-forms $T$ whose set $\mathcal{L}_{T}$ of
singular lines is in some manner special, perhaps there exists $T\in
\operatorname{Alt}(\times^{3}V_{2m})$ such that \emph{through each point
}$\langle a\rangle\in\mathbb{P}V_{2m}$ \emph{there passes precisely one }%
$T$\emph{-singular line? }That is, for some $T\in\operatorname{Alt}(\times
^{3}V_{2m}),$ \emph{perhaps }$\mathcal{L}_{T}$\emph{ is a spread of lines in
}$\mathbb{P}V_{2m}$? Certainly, in finite geometry circles, line-spreads in
$\operatorname*{PG}(2m-1,q)$ are of continuing interest. They also exist in
great profusion, even in low dimension; in particular in $\operatorname*{PG}%
(5,2)$ there exist, see \cite{MatevaTopalova}, 131,044 inequivalent line-spreads!

In the present section we will show that in the case of $V_{6}=V(6,q)$ there
exists $T\in\operatorname{Alt}(\times^{3}V_{6})$ such that $\mathcal{L}_{T}$
is a line-spread in $\operatorname*{PG}(5,q).$ To this end consider a space
$V(3,q^{2})$ with basis $\prec e_{1},e_{2},e_{3}\succ.$ Choose any element
$\rho\in\operatorname*{GF}(q^{2})\setminus\operatorname*{GF}(q)$ and define%
\begin{equation}
e_{4}=\rho e_{1},\quad e_{5}=\rho e_{2},\quad e_{6}=\rho e_{3}. \label{e4e5e6}%
\end{equation}
Then we may view $V(3,q^{2})$ as a 6-dimensional vector space over
$\operatorname*{GF}(q)$:
\begin{equation}
V_{6}=V(6,q)=\prec e_{1},e_{2},e_{3},e_{4},e_{5},e_{6}\succ. \label{V3 as V6}%
\end{equation}
The $q^{4}+q^{2}+1$ points $\langle a\rangle$ of $\operatorname*{PG}%
(2,q^{2})=\mathbb{P}(V(3,q^{2}))$ give rise over $\operatorname*{GF}(q)$ to a
Desarguesian spread $\mathcal{L}$ of $q^{4}+q^{2}+1$ lines $\langle a,\rho
a\rangle$ in $\operatorname*{PG}(5,q)=\mathbb{P}(V(6,q)).$ We aim to show that
$\mathcal{L=L}_{T}$ for some $T\in\operatorname{Alt}(\times^{3}V_{6}).$

If $\tau$ is any element of the 1-dimensional $\operatorname*{GF}(q^{2}%
)$-space $\operatorname{Alt}(\times^{3}V(3,q^{2}))$ then we may define an
element $T\in\operatorname{Alt}(\times^{3}V_{6})$ by%
\begin{equation}
T(x,y,z)=\operatorname*{Tr}(\tau(x,y,z)). \label{T as Tr}%
\end{equation}
Here we use $\operatorname*{Tr}$ to denote the trace $\operatorname*{Tr}%
_{\operatorname*{GF}(q^{2})/\operatorname*{GF}(q)}$ over the subfield
$\operatorname*{GF}(q)$ defined, see \cite[Section 2.3]{LidlNiederreiter}, by%
\begin{equation}
\operatorname*{Tr}(\beta)=\beta+\beta^{q},\quad\beta\in\operatorname*{GF}%
(q^{2}). \label{Defn of Tr}%
\end{equation}
(So $\operatorname*{Tr}$ here is not the absolute trace over the prime
subfield $\operatorname*{GF}(p)$ except if $q=p.)$ Thus defined,
$\operatorname*{Tr}$ is a $\operatorname*{GF}(q)$-linear mapping
$\operatorname*{GF}(q^{2})\rightarrow\operatorname*{GF}(q).$ whence $T$ is
indeed an element of $\operatorname{Alt}(\times^{3}V_{6}).$

Let us fix $\tau$ by requiring $\tau(e_{1},e_{2},e_{3})=\beta$ for some choice
of nonzero element $\beta\in\operatorname*{GF}(q^{2}).$ Upon defining
$c_{i}\in\operatorname*{GF}(q^{2})$ by
\begin{equation}
c_{0}=\operatorname*{Tr}(\beta),~c_{1}=\operatorname*{Tr}(\beta\rho
),~c_{2}=\operatorname*{Tr}(\beta\rho^{2}),~c_{3}=\operatorname*{Tr}(\beta
\rho^{3}) \label{ci = Tr}%
\end{equation}
it follows from (\ref{e4e5e6}), (\ref{T as Tr}) that $t$ in
(\ref{Iso t with T}) is given by%
\begin{equation}
c_{123}=c_{0},~c_{234}=-c_{135}=c_{126}=c_{1},~c_{156}=-c_{246}=c_{345}%
=c_{2},~c_{456}=c_{3}, \label{cijk = c}%
\end{equation}
with $c_{ijk}=0$ for other $i<j<k.$ That is
\begin{equation}
t=c_{0}t_{0}+c_{1}t_{1}+c_{2}t_{2}+c_{3}t_{3}, \label{t = c0t0 + ...}%
\end{equation}
where the trivectors $t_{0},t_{1},t_{2},t_{3}$ are defined by%
\begin{equation}
t_{0}=f_{123},~t_{1}=f_{234}-f_{135}+f_{126},~t_{2}=f_{156}-f_{246}%
+f_{345},~t_{3}=f_{456}. \label{t0 = ,t1 = ...}%
\end{equation}

\medskip

The multiplicative group $\operatorname*{GF}(q^{2})^{\times}$ of the field
$\operatorname*{GF}(q^{2})$ is a cyclic group $\langle\zeta\rangle$ generated
by an irreducible element $\zeta\in\operatorname*{GF}(q^{2})$ of order
$q^{2}-1$: $\zeta^{(q-1)(q+1)}=1.$ The multiplicative group
$\operatorname*{GF}(q)^{\times}$ of the subfield $\operatorname*{GF}(q)$ is
the cyclic group $\langle\xi\rangle$ generated by $\xi=\zeta^{q+1},$ of order
$q-1.$ In making a specific choice of the field elements $\beta,\rho$ in
(\ref{ci = Tr}), it will help to consider separately the cases of odd $q$ and
even $q.$

\subsection{The case of odd $q$ \label{SubSec odd q}}

Suppose that $q=2k+1$ is odd. Then
\begin{equation}
\operatorname*{GF}(q^{2})^{\times}=\langle\zeta\rangle,\quad\text{ where}%
\quad\text{ }\zeta^{4k(k+1)}=1,\quad\zeta^{2k(k+1)}=-1. \label{powers of zeta}%
\end{equation}
In (\ref{e4e5e6}) let us make the following choice of $\rho\notin
\operatorname*{GF}(q)$:%
\begin{equation}
\rho=\zeta^{k+1},\quad\text{ and so}\quad\text{ }\rho^{4k}=1,\quad\rho
^{2k}=-1. \label{choice of rho}%
\end{equation}
It follows that%
\begin{equation}
\operatorname*{Tr}(\rho)=\rho(1+\rho^{2k})=0,\quad\operatorname*{Tr}(\rho
^{2})=\rho^{2}(1+\rho^{4k})=2\rho^{2},\quad\operatorname*{Tr}(\rho^{3}%
)=\rho^{3}(1+\rho^{6k})=0. \label{Tr of powers of rho}%
\end{equation}
Since $\operatorname*{Tr}(1)=2,$ if we make the choice $\beta=\frac{1}{2}$
then the trivector $t$ in (\ref{t = c0t0 + ...}) is%

\begin{equation}
t=f_{123}+\mu(f_{156}-f_{246}+f_{345}),\quad\text{ where}\quad\mu=\rho^{2}.
\label{t = f123 + mu...}%
\end{equation}
Observe that $\mu$ is the square of an element $\rho$ of $\operatorname*{GF}%
(q^{2}),$ but \emph{that }$\mu$\emph{ is an element of }$\operatorname*{GF}%
(q)$\emph{ which is one of the non-squares in }$\operatorname*{GF}(q).$ By
making different choices of the irreducible element $\zeta$ in the definition
(\ref{choice of rho}) of $\rho$ we can arrange for $\mu$ in
(\ref{t = f123 + mu...}) to be any of the non-square elements in
$\operatorname*{GF}(q).$

\begin{theorem}
\label{Thm Line-spread q odd}If $V_{6}=V(6,q)$ where $q$ is odd, consider the
3-form $T\in\operatorname{Alt}(\times^{3}V_{6})$ given by the dual trivector%
\[
t=f_{123}+\mu(f_{156}-f_{246}+f_{345}).
\]
Then, provided only that $\mu\in\operatorname*{GF}(q)$ is chosen to be a
non-square, $\mathcal{L}_{T}$ is a Desarguesian line-spread in
$\operatorname*{PG}(5,q).$
\end{theorem}

\begin{proof}
Because in (\ref{T as Tr}) we have $\tau(a,\rho a,z)=0$ for all $z,$ the
$q^{4}+q^{2}+1$ lines $\langle a,\rho a\rangle$ in $\operatorname*{PG}(5,q)$
are certainly all $T$-singular. To complete the proof we need to show that no
other lines $\langle a,b\rangle$ in $\operatorname*{PG}(5,q)$ are
$T$-singular. Suppose to the contrary that $b\notin\langle a,\rho a\rangle$
yet $T(a,b,x)=0$ for all $x\in V_{6}.$ Now $\tau(Aa,Ab,Ax)=\tau(a,b,x)$ for
any $A\in\operatorname{SL}(3,q^{2}).$ So if $b\notin\langle a,\rho a\rangle$
we can choose $A\in\operatorname{SL}(3,q^{2})$ such that $Aa=e_{1},~Ab=e_{2},$
Since $\operatorname*{Tr}(\tau(e_{1},e_{2},x))\neq0$ if $x=e_{3}$ it follows
that $T(a,b,x)\neq0$ for all $x\in V_{6}.$
\end{proof}

\subsection{The case of even $q$ \label{SubSec even q}}

If $q=2^{h}$ then every element $\mu\in\operatorname*{GF}(q)$ is a square and
for no choice of $\mu$ in (\ref{t = f123 + mu...}) is $\mathcal{L}_{T}$ a
spread. For example, if $\mu=1$ then every line in the plane $\langle
e_{1}+e_{4},e_{2}+e_{5},e_{3}+e_{6}\rangle$ is $T$-singular. However, as we
now demonstrate, $\mathcal{L}_{T}$ is a line-spread for a different choice of
$T$.

In proving this, in addition to the previous $\operatorname*{GF}(q)$-linear
mapping $\operatorname*{Tr}:$ $\operatorname*{GF}(q^{2})\rightarrow
\operatorname*{GF}(q),$ we will also make use of the $\operatorname*{GF}%
(2)$-linear mapping $\operatorname{tr}:$ $\operatorname*{GF}(q)\rightarrow
\operatorname*{GF}(2),$ where $\operatorname{tr}(\mu)\in\operatorname*{GF}(2)$
is the absolute trace of $\mu\in\operatorname*{GF}(2^{h})$:%
\begin{equation}
\operatorname{tr}(\mu)=\mu+\mu^{2}+\mu^{4}+\ldots+\mu^{2^{h-1}}.
\label{tr(mu)}%
\end{equation}
Observe that $\operatorname*{GF}(q)=\mathcal{T}_{0}\cup\mathcal{T}_{1}$ where
\begin{equation}
\mathcal{T}_{i}:=\{\mu\in\operatorname*{GF}(q)|\operatorname{tr}%
(\mu)=i\},\quad i\in\{0,1\}; \label{T0T1}%
\end{equation}
in particular $\mathcal{T}_{0}$ is the kernel of the linear mapping
$\operatorname{tr},$ and is a hyperplane in the $\operatorname*{GF}(2)$-space
$\operatorname*{GF}(q).$ It is easy to see also that $\mathcal{T}%
_{0}=\operatorname{im}F,$ where $F$ denotes the linear endomorphism of the
$\operatorname*{GF}(2)$-space $\operatorname*{GF}(q)$ defined by
$F(\lambda)=\lambda+\lambda^{2}.$ Consequently $\mathcal{T}_{1}$ consists of
those $\mu\in\operatorname*{GF}(q)$ \emph{not} expressible as $\mu
=\lambda+\lambda^{2}$ for any $\lambda\in\operatorname*{GF}(q).$

\begin{lemma}
\label{Lemma Tr = 1,1,...}There exists $\rho\in\operatorname*{GF}%
(q^{2})\setminus\operatorname*{GF}(q),~q=2^{h},$ such that%
\begin{align}
\text{(i) }h\text{ odd}  &  \text{:\qquad}\operatorname*{Tr}(\rho
)=1,~\operatorname*{Tr}(\rho^{2})=1,~\operatorname*{Tr}(\rho^{3}%
)=0;\label{Tr of powers of rho (i)}\\
\text{(ii) }h\text{ even}  &  \text{:\qquad}\operatorname*{Tr}(\rho
)=1,~\operatorname*{Tr}(\rho^{2})=1,~\operatorname*{Tr}(\rho^{3})=\mu,\text{
where }\operatorname{tr}(\mu)=1. \label{Tr of powers of rho (ii)}%
\end{align}
Moreover in (ii) we can choose $\rho$ so that $\mu$ is any pre-assigned
element of $\mathcal{T}_{1}.$
\end{lemma}

\begin{proof}
For any $\zeta\in\operatorname*{GF}(q^{2})\setminus\operatorname*{GF}(q)$ we
have $\ \operatorname*{Tr}(\zeta)\neq0.$ So, since $\operatorname*{Tr}%
(\alpha\zeta)=\alpha\operatorname*{Tr}(\zeta)$ for $\alpha\in
\operatorname*{GF}(q),$ $\zeta\in\operatorname*{GF}(q^{2}),$ then
$\rho=(\operatorname*{Tr}\zeta)^{-1}\zeta$ achieves $\operatorname*{Tr}%
(\rho)=1.$ It then follows that $\operatorname*{Tr}(\rho^{2})=\rho^{2}%
+\rho^{2q}=(\rho+\rho^{q})^{2}=1.$ It further follows that $\rho^{3q}=\rho
^{q}\rho^{2q}=(1+\rho)(1+\rho^{2}),$ whence%
\begin{equation}
\mu:=\operatorname*{Tr}(\rho^{3})=\rho^{3}+\rho^{3q}=1+\rho+\rho^{2}.
\label{mu = 1 + rho +}%
\end{equation}
Now from $\mu+1=\rho+\rho^{2}$ we obtain%
\begin{align}
\operatorname{tr}(\mu+1)  &  =(\rho+\rho^{2})+(\rho^{2}+\rho^{4})+\ldots
(\rho^{2^{h-1}}+\rho^{2^{h}})\nonumber\\
&  =\rho+\rho^{q}=\operatorname*{Tr}(\rho)=1. \label{tr(mu + 1)}%
\end{align}
Suppose first that $h$ is odd. Then $\operatorname{tr}(1)=1$ and so $\mu
\in\mathcal{T}_{0}.$ Now for any $\alpha\in\operatorname*{GF}(q)$ consider
$\rho^{\prime}:=\rho+\alpha.$ Then $\operatorname*{Tr}(\rho^{\prime})=1,$ and
hence $\operatorname*{Tr}((\rho^{\prime})^{2})=1.$ Further if $\mu^{\prime
}:=\operatorname*{Tr}((\rho^{\prime})^{3})$ then
\begin{equation}
\mu^{\prime}=\operatorname*{Tr}(\rho^{3})+\alpha\operatorname*{Tr}(\rho
^{2})+\alpha^{2}\operatorname*{Tr}(\rho)+\alpha^{3}\operatorname*{Tr}%
(1)=\mu+\alpha+\alpha^{2}, \label{mumualphaalpha}%
\end{equation}
whence $\mu^{\prime}=0$ for a suitable choice of $\alpha,$ thus achieving
(\ref{Tr of powers of rho (i)}).

If instead $h$ is even and so $\mu\in\mathcal{T}_{1},$ then we see from
(\ref{mumualphaalpha}) that we can achieve (\ref{Tr of powers of rho (ii)})
for any pre-assigned $\mathcal{\mu\in T}_{1}.$
\end{proof}

\begin{theorem}
\label{Thm Line-spread q even}If $V_{6}=V(6,q)$ where $q=2^{h}$ is even, then
if $h$ is odd the 3-form $T\in\operatorname{Alt}(\times^{3}V_{6})$ given by
the trivector%
\begin{equation}
t=f_{234}+f_{135}+f_{126}+f_{156}+f_{246}+f_{345} \label{t = 234 = ... + 345}%
\end{equation}
is such that $\mathcal{L}_{T}$ is a Desarguesian line-spread in
$\operatorname*{PG}(5,q).$ If $h$ is even then, for any $\mu\in
\operatorname*{GF}(q)$ satisfying $\operatorname{tr}(\mu)=1,$ the 3-form
$T\in\operatorname{Alt}(\times^{3}V_{6})$ given by the trivector%
\begin{equation}
t=f_{234}+f_{135}+f_{126}+f_{156}+f_{246}+f_{345}+\mu f_{456}
\label{t = 234 = ... + 345 + mu}%
\end{equation}
is such that $\mathcal{L}_{T}$ is a Desarguesian line-spread in
$\operatorname*{PG}(5,q).$
\end{theorem}

\begin{proof}
In (\ref{ci = Tr}) choose $\rho$ as in Lemma \ref{Lemma Tr = 1,1,...} and
choose $\beta=1,$ and so $c_{0}=\operatorname*{Tr}(1)=0.$ Then we obtain
(\ref{t = 234 = ... + 345}) and (\ref{t = 234 = ... + 345 + mu}) from
(\ref{t = c0t0 + ...}) (\ref{t0 = ,t1 = ...}). The rest of the proof is as in
the proof of Theorem \ref{Thm Line-spread q odd}.
\end{proof}

\begin{remark}
The canonical form (\ref{t = 234 = ... + 345}) was obtained previously in
\cite{ShawSpreads}, albeit only in the special case $q=2.$ In
\cite{ShawSpreads} two alternative canonical forms were also found, namely
$t^{\prime}$ and $t^{\prime\prime}$ as given by%
\begin{align}
t^{\prime}  &  =f_{156}+f_{246}+f_{345}+f_{123}+f_{456}\nonumber\\
t^{\prime\prime}  &  =f_{234}+f_{135}+f_{126}+f_{123}+f_{456}~. \label{t't''}%
\end{align}
\emph{These are also alternatives to (\ref{t = 234 = ... + 345}) for any}
$q=2^{h}$, $h$\emph{ odd.} One way to obtain these alternatives is by use of
the choice $\rho=\zeta^{(q-1)(q+1)/3}.$in (\ref{ci = Tr}). For if $h$ is odd
then $3|(q+1)$ and so $\zeta^{(q+1)/3}\in\operatorname*{GF}(q),$ whence
$\rho\notin\operatorname*{GF}(q).$ Since $\rho^{3}=1,$ we have
$\operatorname*{Tr}(\rho^{2})=\operatorname*{Tr}(\rho)$ and
$\operatorname*{Tr}(\rho^{3})=0.$ So in (\ref{ci = Tr}) the choices $\beta=1$,
$\beta=\rho,~\beta=\rho^{2}$ give rise respectively to the trivectors $t$ (as
in (\ref{t = 234 = ... + 345}))$,t^{\prime},t^{\prime\prime}.$
\end{remark}

\begin{remark}
For even $n$ one might wonder whether $\mathcal{L}_{T}$ can be a line spread
for fields other than finite fields. Indeed, suppose that $\mathbb{F}$ is
algebraically closed; is it possible that $\mathcal{L}_{T}$ is a spread? The
$T$'s for which this is the case can be shown to form a Zariski-open subset of
$\operatorname{Alt}(\times^{3}V_{n}),$ but this subset may be empty. Indeed,
we have performed Gr\"{o}bner basis calculations which show that for $n=4$ and
$n=6$ and for algebraically closed $\mathbb{F}$ of characteristic zero there
are no trilinear forms $T$ for which $\mathcal{L}_{T}$ is a spread. For $n=8$
our computational approach is not feasible, and new ideas will be needed.
\end{remark}

\section{Some other noteworthy alternating 3-forms \label{Sec some others}}

So far we have been looking at alternating 3-forms $T$ for which
$\mathcal{L}_{T}$ is as small a set as possible. In contrast we now give some
examples of interesting alternating 3-forms $T\in\operatorname{Alt}(\times
^{3}V_{n})$ which are non-degenerate yet for which some sizeable subspace
$V_{r}$ of $V_{n}$ is such that \emph{every line in }$\mathbb{P}V_{r}$\emph{
is }$T$\emph{-singular.} Let us term such a subspace $V_{r}$ \emph{totally
}$T$\emph{-singular.} Our first example is in dimension $n=6,$ where
\begin{equation}
t=f_{156}+f_{246}+f_{345} \label{t in dim 6}%
\end{equation}
is the trivector of a non-degenerate $T\in\operatorname{Alt}(\times^{3}V_{6})$
for which, for any field $\mathbb{F},$ the $3$-space $V_{3}=\prec e_{1}%
,e_{2},e_{3}\succ$ is totally $T$-singular. It is easy to check that no
subspace of dimension $>3$ is totally singular. A second example is in
dimension $n=10,$ where, writing $\mathtt{x}=10,$
\begin{equation}
t=f_{17\mathtt{x}}+f_{28\mathtt{x}}+f_{39\mathtt{x}}+f_{489}+f_{579}+f_{678}
\label{t in dim 10}%
\end{equation}
is the trivector of a non-degenerate $T\in\operatorname{Alt}(\times^{3}%
V_{10})$ for which the $6$-space $V_{6}=$ $\prec e_{1},~...~,e_{6}\succ$ is
totally $T$-singular.

\begin{theorem}
\label{Thm TT-S}If $n=\frac{1}{2}s(s+1),s>2,$ then, for any field
$\mathbb{F},$ there exists a single $\operatorname*{GL}(n,\mathbb{F})$-orbit,
say $\Omega,$ of non-degenerate alternating 3-forms $T$ on $V_{n}$ with the
property that there is a unique totally $T$-singular subspace $V_{r}$ of
$V_{n}$ of dimension $r=\frac{1}{2}s(s-1).$
\end{theorem}

\begin{proof}
See \cite[Section 4]{DraismaShaw}.
\end{proof}

\begin{remark}
In dimension $n=15$ the subspace $V_{r}$ in the theorem is of dimension $10.$
By use of the crude inequality $|\operatorname*{GL}(n,q)|\ll q^{n^{2}},$ as in
Section \ref{Sec Intro}, we get the even cruder lower bound $q^{230}$ for the
number of $\operatorname*{GL}(15,q)$ orbits of alternating 3-forms $T$ on
$V(15,q).$ Clearly there is no possibility of studying all of these zillions
of orbits! But perhaps the orbit $\Omega$ does deserve further study?
\end{remark}

%

\small

\medskip

\noindent{\small Jan Draisma }

{\small Department of Mathematics and Computer Science }

{\small Technische Universiteit Eindhoven }

{\small P.O. Box 513, 5600 MB Eindhoven }

({\small also CWI Amsterdam) }

{\small The Netherlands}

{\small e-mail: }\texttt{j.draisma@tue.nl}

\medskip

\noindent{\small Ron \thinspace Shaw }

{\small Centre for Mathematics }

{\small University of Hull }

{\small Hull HU6 7RX }

{\small United Kingdom}

{\small e-mail: }\texttt{r.shaw@hull.ac.uk}

\end{document}